\documentclass{amsart}
\usepackage[bookmarks=true,colorlinks,urlcolor=black,citecolor=black,linkcolor=black]{hyperref}
\usepackage{mathrsfs}
\usepackage{amsmath}
\usepackage{mathtools}
\usepackage{amssymb}
\usepackage{graphicx}
\usepackage{latexsym}
\newtheorem{theorem}{Theorem}
\newtheorem{lemma}{Lemma}
\theoremstyle{definition}
\newtheorem*{acknowledgements}{Acknowledgements}

\theoremstyle{definition}

\graphicspath{{./Figures/}{./30baiIMO2016/}{C:/Users/hung/Dropbox/Hung/Figures/}}
\begin{document}
	\title[Two generalizations of the ButterFly Theorem]{Two generalizations of the ButterFly Theorem}
	
	\author[Tran Quang Hung]{Tran Quang Hung}
	\address{High School for Gifted Students, Hanoi, Vietnam}
	\email{tranquanghung@hus.edu.vn}
	
	\author[Luis Gonz\'alez]{Luis Gonz\'alez}
	\address{Engineer at Houston, Texas}
	\email{luis240985@gmail.com}
	
	\maketitle

\begin{abstract}We establish two direct extensions to the Butterfly Theorem on the cyclic quadrilateral along with the proofs using the projective method and analytic geometry of the Cartesian coordinate system.
\end{abstract}

\section{Introduction}

We repeat the Butterfly Theorem expressed with the chord of the circle; see \cite{1,2,3,4,4a,4b}. This is an interesting and important theorem of plane Euclidean geometry. This classic theorem also has a lot of solutions; see \cite{1,2,3,4,4b}. Previously, the first author of this article also contributed a new proof to this theorem in \cite{5a}.

\begin{theorem}[Butterfly Theorem]Let $M$ be the midpoint of a chord $AB$ of a circle $\omega$, through $M$ two other chords $CD$ and $EF$ of $\omega$ are drawn. If $C$ and $F$ are on opposite sides of $AB$, and $CF$, $DE$ intersect $AB$ at $G$ and $H$ respectively, then $M$ is also the midpoint of $GH$.
\end{theorem}
\begin{figure}[htbp]
	\begin{center}\scalebox{0.7}{\includegraphics{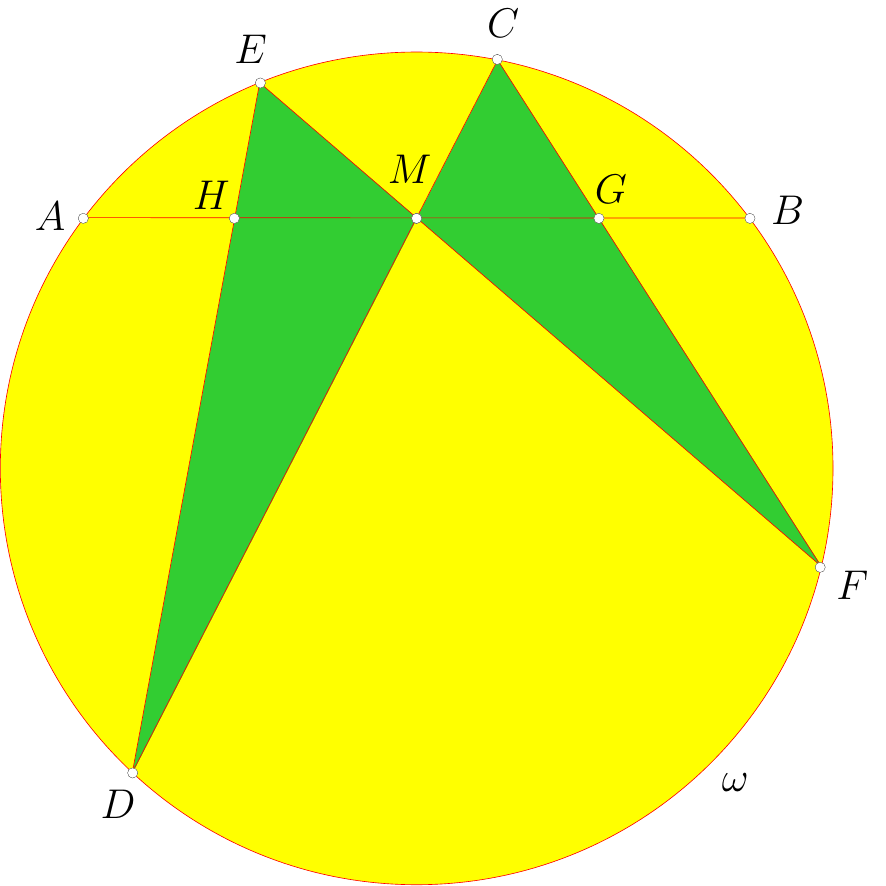}}\end{center}	
	\label{fig00}
	\caption{Butterfly Theorem for chord of circle}
\end{figure}

If $O$ is the center of $\omega$, the condition $M$ is the midpoint of $AB$ and can be changed into $AB$ is perpendicular to $OM$. Based on this property we may have a different version of Butterfly Theorem as follows

\begin{theorem}[Butterfly Theorem for cyclic quadrilateral]\label{thm0}Let $ABCD$ be a quadrilateral inscribed in circle $\omega$. Let $O$ be the center of $\omega$. Diagonals $AC$ and $BD$ meet at $P$. Perpendicular line from $P$ to $OP$ meets the side line $AB$ and $CD$ at $Q$ and $R$, respectively. Then, $P$ is the midpoint of $QR$.
\end{theorem}
\begin{figure}[htbp]
	\begin{center}\scalebox{0.7}{\includegraphics{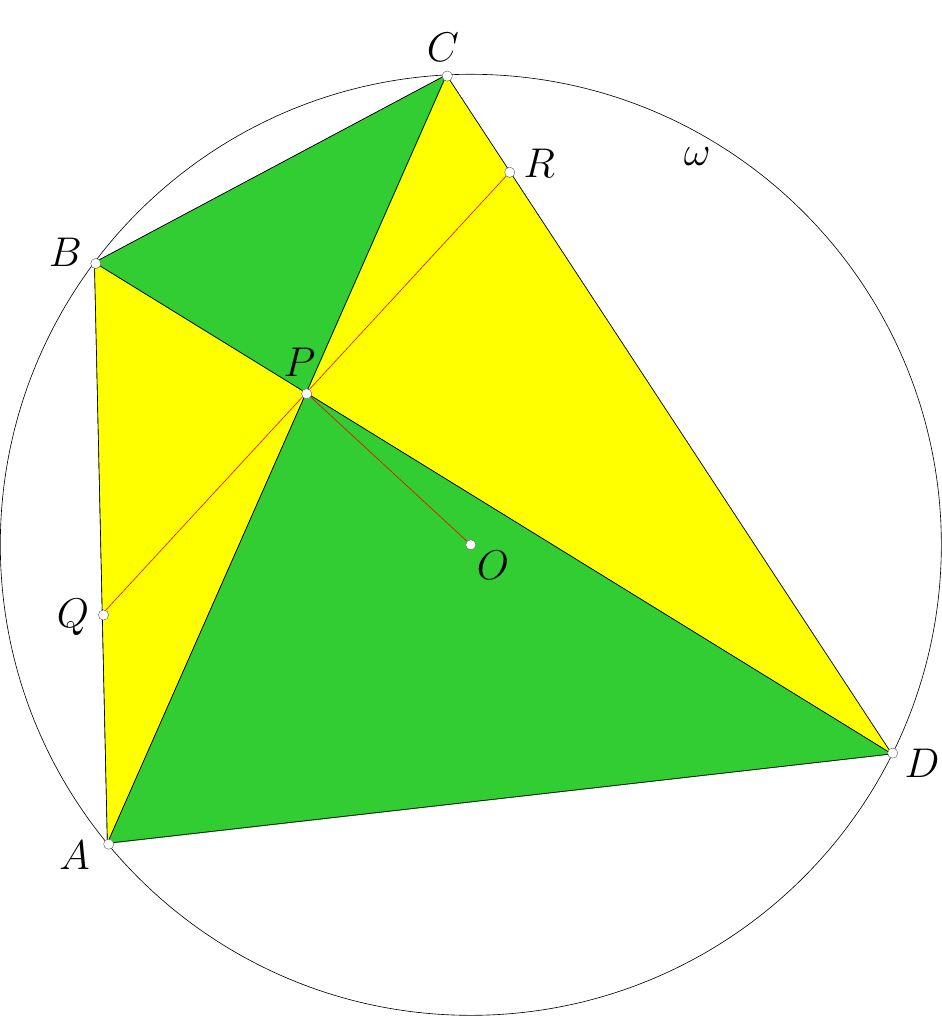}}\end{center}\label{fig0}
	\caption{Butterfly Theorem for cyclic quadrilateral}
\end{figure}

The Butterfly Theorem has a lot of extensions and generalizations, see \cite{4b,4c,6,7,8,9}. In this paper, we shall generalize the Theorem \ref{thm0} by replacing the cyclic quadrilateral into any quadrilateral. We shall have two generalizations of Theorem \ref{thm0} by this way. Where the quadrilateral is cyclic, we obtain Theorem \ref{thm0}. We also refer to the name as well as the properties of Quadrangle, Quadrilateral, and Quadrigon in \cite {4d,4e}.

\begin{theorem}[The first generalization of Butterfly Theorem]\label{thm1}Let $ABCD$ be an arbitrary quadrilateral. Diagonals $AC$ and $BD$ meet at $P$. Let $O_a$, $O_b$, $O_c$, and $O_d$ be the circumcenters of the triangles $BCD$, $CDA$, $DAB$, and $ABC$, respectively. Let $M$ and $N$ be the midpoints of segments $O_aO_c$ and $O_bO_d$, respectively. The perpendicular lines from $P$ to $MN$ meets lines $AB$ and $CD$ at $Q$ and $R$, respectively. Then, $P$ is the midpoint of $QR$.
\end{theorem}
\begin{figure}[htbp]
	\begin{center}\scalebox{0.7}{\includegraphics{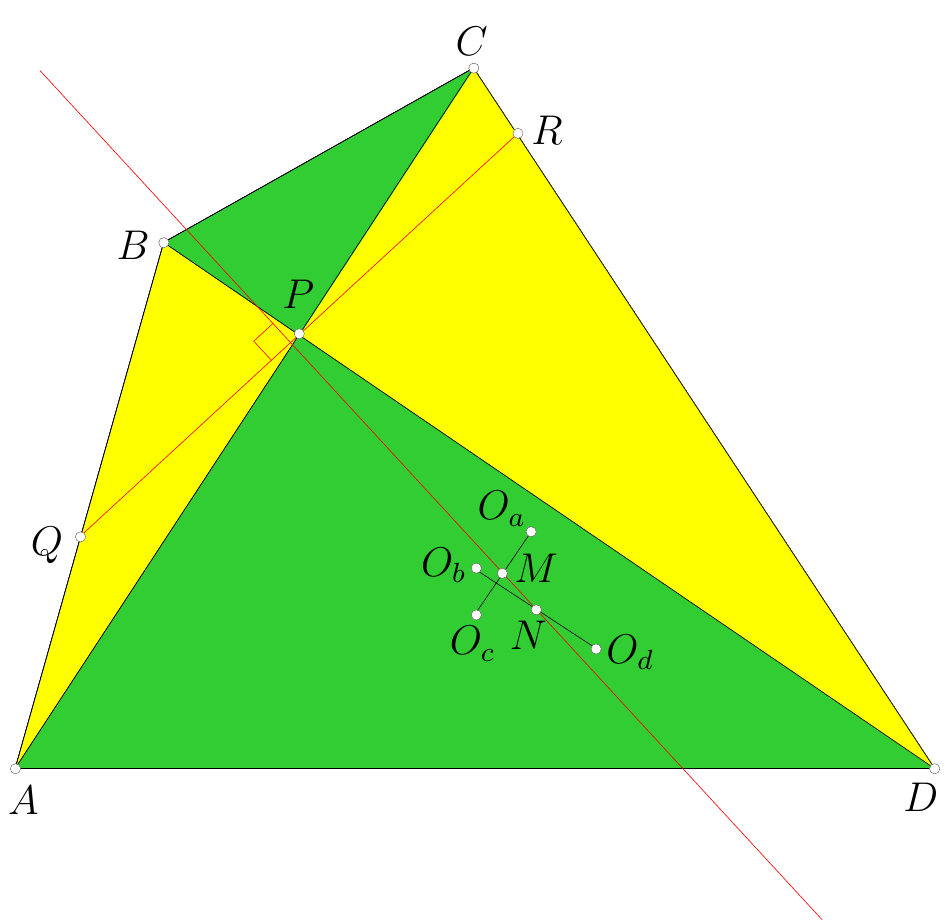}}\end{center}	
	\label{fig1}
	\caption{The first generalization of Butterfly Theorem}
\end{figure}

\begin{theorem}[The second generalization of Butterfly Theorem]\label{thm2}Let $ABCD$ be an arbitrary quadrilateral. Diagonals $AC$ and $BD$ meet at $P$. Perpendicular bisectors of two segments $AC$ and $BD$ meet at $X$. Perpendicular bisectors of two segments $AB$ and $CD$ meet at $Y$. Perpendicular bisectors of two segments $AD$ and $BC$ meet at $Z$. Construct a parallelogram $XYWZ$. Perpendicular line to $PW$ passes through $P$ which meets the sides $AD$ and $BC$ at $Q$ and $R$, respectively. Then, $P$ is the midpoint of $QR$.
\end{theorem}
\begin{figure}[htbp]
	\begin{center}\scalebox{0.7}{\includegraphics{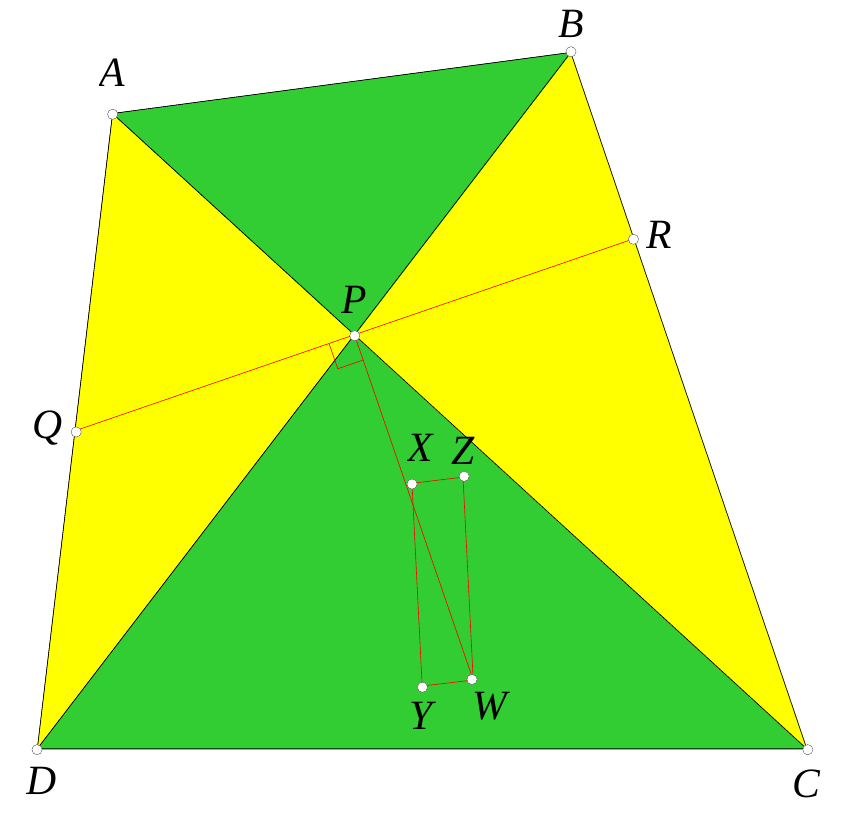}}\end{center}	
	\label{fig2}
	\caption{The second generalization of Butterfly Theorem}
\end{figure}
We give two proofs to both Theorem \ref{thm1} and Theorem \ref{thm2}. One of which is an elegant proof using projective geometry. For the remaining solutions, we use analytic geometry of the Cartesian coordinate system, if we have chosen the appropriate coordinate system, then we prove both theorems by other concise numerical methods. Obviously, the selection of the appropriate axis system is also a very interesting illustration showing the power of using the Cartesian coordinate system in proving complex geometrical theorems without using powerful tools of the number. The idea of Ren\'e Descartes in coordinate geometry is really very great; see \cite{0,0a}.

\section{Proofs of Theorems}
	
For the first proof of Theorem \ref{thm1}, we shall use the following lemmas

\begin{lemma}[AMM Problem 12147 \cite{4}]\label{lem1}$ABCD$ is an arbitrary quadrilateral. Perpendicular bisectors of $AB$, $CD$ meet at $X$ and the perpendicular bisectors of $BC$, $DA$ meet at $Y$. Then $XY$ is perpendicular to the Newton line of $ABCD$.
\end{lemma}

\begin{lemma}\label{lem2}Let $ABCD$ and $PQRS$ be two quadrilaterals in the plane such that $PQ\perp AB$, $QR \perp BC$, $RS \perp CD$, $SP \perp DA$, $PR \perp BD$, $SQ\perp AC$, and $PR\perp BD$. Let $J$ be the intersection point of lines $PQ$ and $RS$. Let $K$ be the intersection point of lines $QR$ and $SP$. Then, $JK$ is perpendicular to the Newton line of $ABCD$.
\end{lemma}

\begin{proof}[The first proof of Theorem \ref{thm1}]Let $F \equiv BC \cap DA$ and $G \equiv CD \cap AB$. Since $AB \perp O_cO_d$, $BC \perp O_dO_a$, $CD \perp O_aO_b$, $DA \perp O_bO_c$, $BD \perp O_aO_c$, and $AC\perp O_bO_d$ then using Lemma \ref{lem2} on quadrilateral $O_aO_bO_cO_d$, we get that $GF$ is perpendicular to its Newton line $MN$ so that $QR \parallel FG$. Since $F(P,A;B,G)=F(P,R;Q,G)=-1$ and $QR \parallel FG$, it follows that $P$ is midpoint of $QR$ (following the properties of harmonic pencil in \cite{3}). We complete the solution.
\end{proof}

We shall use the Cartesian coordinate system for the second proof of Theorem \ref{thm1}.
\begin{figure}[htbp]
	\begin{center}\scalebox{0.7}{\includegraphics{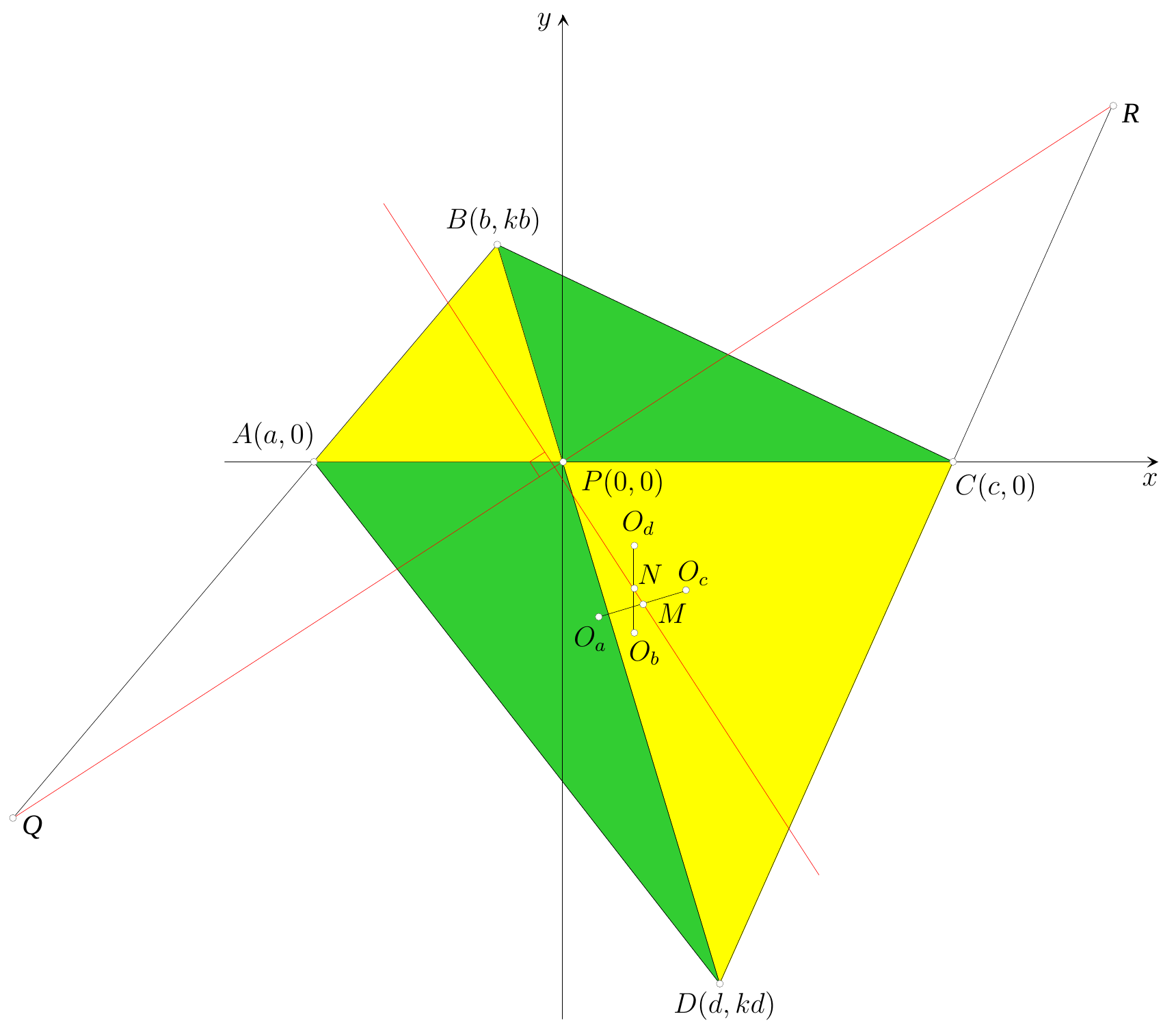}}\end{center}	
	\label{fig6}
	\caption{The second proof of Theorem \ref{thm1} using Cartesian coordinate system}
\end{figure}
\begin{proof}[The second proof of Theorem \ref{thm1}]Choosing Cartesian coordinate system such that $P(0,0)$, $A(a,0)$, $B(b,kb)$, $C(c,0)$, and $D(d,kd)$, it is easily seen that $P$ is intersection of two diagonals $AC$ and $BD$. We now compute the coordinates of some points.

Circumcenter $O_a$ of triangle $BCD$ is intersection of perpendicular bisector of $BC$, $BD$
\begin{equation}\label{eq1}
O_a = \left(\frac{b d k^{2} + b d + c^{2}}{2 c}, \frac{b c k^{2} + b c - b d k^{2} - b d - c^{2} + c d k^{2} + c d}{2 c k} \right).
\end{equation}
Circumcenter $O_b$ of triangle $CDA$ is intersection of perpendicular bisector of $CD$, $CA$
\begin{equation}\label{eq2}
O_b = \left(\frac{a+c}{2}, \frac{a c - a d - c d + d^{2} k^{2} + d^{2}}{2 d k} \right).
\end{equation}
Circumcenter $O_c$ of triangle $DAB$ is intersection of perpendicular bisector of $DA$, $DB$
\begin{equation}\label{eq3}
O_c = \left(\frac{a^{2} + b d k^{2} + b d}{2 a}, \frac{-a^{2} + a b k^{2} + a b + a d k^{2} + a d - b d k^{2} - b d}{2 a k} \right).
\end{equation}
Circumcenter $O_d$ of triangle $ABC$ is intersection of perpendicular bisector of $AB$, $AC$
\begin{equation}\label{eq4}
O_d = \left(\frac{a+c}{2}, \frac{-a b + a c + b^{2} k^{2} + b^{2} - b c}{2 b k} \right).
\end{equation}
From \eqref{eq1}, \eqref{eq2}, \eqref{eq3}, and \eqref{eq4}, midpoints $M$ and $N$ of $O_aO_c$ and $O_bO_d$, respectively, have the coordinates
\begin{equation}\label{eq5}M = \left(\frac{a^{2} c + a b d k^{2} + a b d + a c^{2} + b c d k^{2} + b c d}{4 a c}, \frac{\splitfrac{-a^{2} c + 2 a b c k^{2} + 2 a b c - a b d k^{2} - a b d}{ - a c^{2} + 2 a c d k^{2} + 2 a c d - b c d k^{2} - b c d}}{4 a c k} \right),
\end{equation}
\begin{equation}\label{eq6}N = \left(\frac{a + c}{2}, \frac{a b c - 2 a b d + a c d + b^{2} d k^{2} + b^{2} d - 2 b c d + b d^{2} k^{2} + b d^{2}}{4 b d k} \right).
\end{equation}
From \eqref{eq5} and \eqref{eq6}, the perpendicular from $P(0,0)$ to line $MN$ has equation
\begin{equation}\label{eq7}y = \frac{-a b d k - b c d k}{a b c - a b d + a c d - b c d} x.
\end{equation}
Lines $BC$ and $AD$ have equation
\begin{equation}\label{eq8}
(AB):\ y = a b \frac{k}{a - b} - b k \frac{x}{a - b}.
\end{equation}
\begin{equation}\label{eq9}
(CD):\ y = c d \frac{k}{c - d} - d k \frac{x}{c - d}.
\end{equation}
From \eqref{eq7}, \eqref{eq8}, and \eqref{eq9}, we obtain the intersection $Q$ and $R$ of the perpendicular from $P$ to line $MN$ with the lines $AB$ and $CD$, respectively, have the coordinates
\begin{equation}\label{eq10}
Q = \left(\frac{-a b c + a b d - a c d + b c d}{a d - b c}, \frac{a b d k + b c d k}{a d - b c} \right),\end{equation}
\begin{equation}\label{eq11}
R = \left(\frac{a b c - a b d + a c d - b c d}{a d - b c}, \frac{-a b d k - b c d k}{a d - b c} \right).\end{equation}
Finally, from \eqref{eq10} and \eqref{eq11}, we easily check $P$ is the midpoint of $QR$. This completes the second proof.
\end{proof}

Using the Cartesian coordinate system again, we give proof of Theorem \ref{thm2}.
\begin{figure}[htbp]
	\begin{center}\scalebox{0.7}{\includegraphics{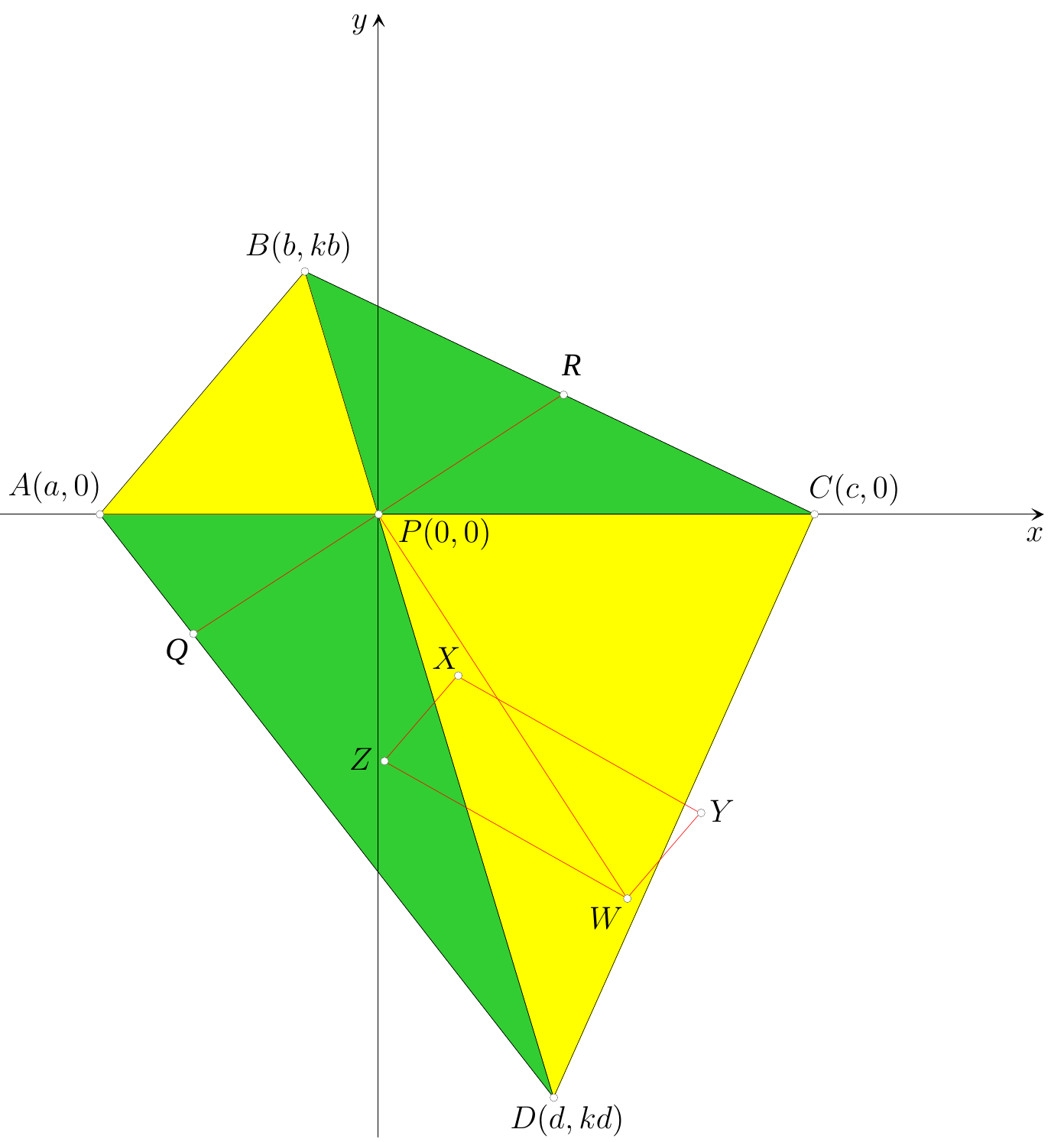}}\end{center}	
	\label{fig7}
	\caption{Proof of Theorem \ref{thm2} using Cartesian coordinate system}
\end{figure}
\begin{proof}[The first proof of Theorem \ref{thm2}]Choosing the Cartesian coordinate system such that $P(0,0)$, $A(a,0)$, $B(b,kb)$, $C(c,0)$, and $D(d,kd)$, we easily seen $P$ is intersection of two diagonals $AC$ and $BD$. We now compute the coordinates of some points.

Perpendicular bisectors of $AC$ and $BD$ meet at $X$ which has coordinates
\begin{equation}\label{eq12}X = \left(\frac{a+c}{2}, \frac{-a + bk^{2} + b - c + dk^{2} + d}{2k} \right).
\end{equation}
Perpendicular bisectors of $AB$ and $CD$ meet at $Y$ which has coordinates
\begin{equation}\label{eq13}Y = \left(\frac{a^{2}d - b^{2}dk^{2} - b^{2}d - bc^{2} + bd^{2}k^{2} + bd^{2}}{2ad - 2bc}, \frac{\splitfrac{a^{2}c - a^{2}d - ac^{2} + ad^{2}k^{2} + ad^{2} - b^{2}ck^{2}}{ - b^{2}c + b^{2}dk^{2} + b^{2}d + bc^{2} - bd^{2}k^{2} - bd^{2}}}{2adk - 2bck} \right).
\end{equation}
Perpendicular bisectors of $AD$ and $BC$ meet at $Z$ which has coordinates
\begin{equation}\label{eq14}Z = \left(\frac{a^{2}b + b^{2}dk^{2} + b^{2}d - bd^{2}k^{2} - bd^{2} - c^{2}d}{2ab - 2cd}, \frac{\splitfrac{-a^{2}b + a^{2}c + ab^{2}k^{2} + ab^{2} - ac^{2}- b^{2}dk^{2}}{ - b^{2}d + bd^{2}k^{2} + bd^{2} + c^{2}d - cd^{2}k^{2} - cd^{2}}}{2abk - 2cdk} \right).
\end{equation}
From \eqref{eq12}, \eqref{eq13}, and \eqref{eq14}, the vertex $W$ of parallelogram $XYWZ$ has coordinates
\begin{equation}\label{eq15}W = \left(f(a,b,c,d,k),g(a,b,c,d,k)\right)
\end{equation}
where
\begin{equation}\label{eq16}f(a,b,c,d,k)=\frac{\splitfrac{a^{3}bd - a^{2}bcd - ab^{3}dk^{2} - ab^{3}d + 2ab^{2}d^{2}k^{2} + 2ab^{2}d^{2} - abc^{2}d - abd^{3}k^{2}}{ - abd^{3} - b^{3}cdk^{2} - b^{3}cd + 2b^{2}cd^{2}k^{2} + 2b^{2}cd^{2} + bc^{3}d - bcd^{3}k^{2} - bcd^{3}}}{2a^{2}bd - 2ab^{2}c - 2acd^{2} + 2bc^{2}d},
\end{equation}
and
\begin{equation}\label{eq17}
g(a,b,c,d,k)=\frac{\begin{aligned}
&a^{3}bc - a^{3}bd + a^{3}cd - 2a^{2}bc^{2} + a^{2}bcd
- 2a^{2}c^{2}d - ab^{3}ck^{2}- ab^{3}c\\ 
&+ ab^{3}dk^{2} + ab^{3}d + ab^{2}cdk^{2} + ab^{2}cd - 2ab^{2}d^{2}k^{2}- 2ab^{2}d^{2} + abc^{3}\\ 
&+ abc^{2}d + abcd^{2}k^{2} + abcd^{2}+ abd^{3}k^{2}+ abd^{3} + ac^{3}d - acd^{3}k^{2} - acd^{3}\\ 
&+ b^{3}cdk^{2} + b^{3}cd- 2b^{2}cd^{2}k^{2} - 2b^{2}cd^{2} - bc^{3}d + bcd^{3}k^{2} + bcd^{3}\end{aligned}}{2a^{2}bdk - 2ab^{2}ck - 2acd^{2}k + 2bc^{2}dk}.
\end{equation}
From \eqref{eq15}, \eqref{eq16}, and \eqref{eq17}, we get equation of perpendicular line from $P$ to line $PW$ as follows
\begin{equation}\label{eq18}
y = \frac{-abdk - bcdk}{abc - abd + acd - bcd}x.
\end{equation}
Thus intersection points of this line with lines $AD$ and $BC$ have coordinates
\begin{equation}\label{eq19}
Q = \left(\frac{-abc + abd - acd + bcd}{ab - cd}, \frac{abdk + bcdk}{ab - cd} \right)
\end{equation}
and
\begin{equation}\label{eq20}
R = \left(\frac{abc - abd + acd - bcd}{ab - cd}, \frac{-abdk - bcdk}{ab - cd} \right).
\end{equation}
From \eqref{eq19} and \eqref{eq20}, we easily check $P$ is the midpoint of $QR$. This completes the proof.
\end{proof}

The second proof of Theorem \ref{thm2} is based on the idea of the first proof of Theorem \ref{thm1}. We need a lemma

\begin{lemma}\label{lem3}Let $ABCD$ be an arbitrary quadrilateral. Diagonals $AC$ and $BD$ meet at $P$. Let $O_a$, $O_b$, $O_c$, and $O_d$ be the circumcenters of the triangles $BCD$, $CDA$, $DAB$, and $ABC$, respectively. Let $M$ and $N$ be the midpoints of diagonals $AC$ and $BD$, respectively. Then, the circles with diameters $O_aO_c$, $O_bO_d$, and the circumcircle of triangle $PMN$ are coaxial.
\end{lemma}

\begin{proof}Like the second proof of Theorem \ref{thm1}, we choose Cartesian coordinate system such that $P(0,0)$, $A(a,0)$, $B(b,kb)$, $C(c,0)$, and $D(d,kd)$, we easily seen $P(0,0)$ is intersection of two diagonals $AC$ and $BD$, and midpoints $M\left(\frac{a+c}{2},0\right)$ and $N\left(\frac{b+d}{2},k\frac{b+d}{2}\right)$. We now compute the coordinates of some points.
	
Circumcenter $O_a$ of triangle $BCD$ is intersection of perpendicular bisector of $BC$, $BD$
	\begin{equation}\label{eq21}
		O_a = \left(\frac{b d k^{2} + b d + c^{2}}{2 c}, \frac{b c k^{2} + b c - b d k^{2} - b d - c^{2} + c d k^{2} + c d}{2 c k} \right).
	\end{equation}
	Circumcenter $O_b$ of triangle $CDA$ is intersection of perpendicular bisector of $CD$, $CA$
	\begin{equation}\label{eq22}
		O_b = \left(\frac{a+c}{2}, \frac{a c - a d - c d + d^{2} k^{2} + d^{2}}{2 d k} \right).
	\end{equation}
	Circumcenter $O_c$ of triangle $DAB$ is intersection of perpendicular bisector of $DA$, $DB$
	\begin{equation}\label{eq23}
		O_c = \left(\frac{a^{2} + b d k^{2} + b d}{2 a}, \frac{-a^{2} + a b k^{2} + a b + a d k^{2} + a d - b d k^{2} - b d}{2 a k} \right).
	\end{equation}
	Circumcenter $O_d$ of triangle $ABC$ is intersection of perpendicular bisector of $AB$, $AC$
	\begin{equation}\label{eq24}
		O_d = \left(\frac{a+c}{2}, \frac{-a b + a c + b^{2} k^{2} + b^{2} - b c}{2 b k} \right).
	\end{equation}
From this, note that power of point $W$ with respect to circle diameter $XY$ is the dot product
$$\mathcal{P}_{W/(XY)}=\overrightarrow{WX}\cdot \overrightarrow{WY}.$$
Using \eqref{eq21}, \eqref{eq22}, \eqref{eq23}, and \eqref{eq24}, we have the ratio of powers
\begin{equation}\label{eq25}
\frac{\mathcal{P}_{P/(O_AO_C)}}{\mathcal{P}_{P/(O_BO_D)}}=\frac{\overrightarrow{PO_A}\cdot\overrightarrow{PO_C}}{\overrightarrow{PO_B}\cdot\overrightarrow{PO_D}}=\frac{(k^2+1)bd}{ac}
\end{equation}
\begin{equation}\label{eq26}
	\frac{\mathcal{P}_{M/(O_AO_C)}}{\mathcal{P}_{M/(O_BO_D)}}=\frac{\overrightarrow{MO_A}\cdot\overrightarrow{MO_C}}{\overrightarrow{MO_B}\cdot\overrightarrow{MO_D}}=\frac{(k^2+1)bd}{ac}
\end{equation}
\begin{equation}\label{eq27}
	\frac{\mathcal{P}_{N/(O_AO_C)}}{\mathcal{P}_{N/(O_BO_D)}}=\frac{\overrightarrow{NO_A}\cdot\overrightarrow{NO_C}}{\overrightarrow{NO_B}\cdot\overrightarrow{NO_D}}=\frac{(k^2+1)bd}{ac}.
\end{equation}
Therefore, 
\begin{equation}\label{eq28}\frac{\overrightarrow{PO_A}\cdot\overrightarrow{PO_C}}{\overrightarrow{PO_B}\cdot\overrightarrow{PO_D}}=\frac{\overrightarrow{MO_A}\cdot\overrightarrow{MO_C}}{\overrightarrow{MO_B}\cdot\overrightarrow{MO_D}}=\frac{\overrightarrow{NO_A}\cdot\overrightarrow{NO_C}}{\overrightarrow{NO_B}\cdot\overrightarrow{NO_D}}=\frac{(k^2+1)bd}{ac}=\frac{\overrightarrow{PB}\cdot\overrightarrow{PD}}{\overrightarrow{PA}\cdot\overrightarrow{PC}}.
\end{equation}
By property of ratio powers, we have the circles with diameters $O_aO_c$, $O_bO_d$, and the circumcircle of triangle $PMN$ are coaxial. This completes the proof.
\end{proof}

\begin{proof}[Second proof of Theorem \ref{thm2}]Letting $U \equiv AD \cap BC$ and $V \equiv AB \cap CD$. In order to prove that $P$ is midpoint of $QR$, we need to prove $QR \parallel UV$ (using harmonic pencil as in the first proof of Theorem \ref{thm1}) or in other words $PW \perp UV$. By homothety with center $X$ and factor $\frac{1}{2}$, this is equivalent to prove that the line joining the midpoints $P'$ and $K'$ of $XP$ and $XK$ is perpendicular to $UV$. From Lemma \ref{lem3}, we deduce that $K'$ falls on the Newton line of $O_AO_BO_CO_C$ and from Lemma \ref{lem3}, we know this Newton line is perpendicular to $UV$ from Lemma \ref{lem2}, which completes the proof.
\end{proof}

\begin{acknowledgements}We thank Chris van Tienhoven and his associates in \cite{4d,4e} for providing us with valuable knowledge about Quadrangle, Quadrilateral, and Quadrigon and how to categorize them.
	
We thank open math software Geogebra Geometry \cite{10a} and Sage Notebook \cite{10} for drawing and transforming algebraic expressions.
\end{acknowledgements}


\begin{thebibliography}{12}
\bibitem{0} René Descartes, Discourse de la M\'ethode (Leiden, Netherlands): Jan Maire, 1637, appended book: La G\'eom\'etrie, book one, p.~299.

\bibitem{0a} Sorell, T.: {\it Descartes: A Very Short Introduction} (2000). New York: Oxford University Press. p.~19.
	
\bibitem{1}	Coxeter, H. S. M. and Greitzer, S. L.: {\it Geometry Revisited}, Washington, DC: Math. Assoc. Amer., 1967., p.~53.

\bibitem{2} Johnson, R. A.: {\it Modern Geometry: An Elementary Treatise on the Geometry of the Triangle and the Circle}, Boston, MA: Houghton Mifflin, 1929., p.~172.

\bibitem{3} Coxeter, H. S. M.: {\it Projective geometry}, Blaisdell, New York, 1964., p.~78.

\bibitem{4} Coxeter, H. S. M.: {\it Non-Euclidean Geometry}, University of Toronto Press, 1942., p.~29.

\bibitem{4a} Weisstein, E. W.: Butterfly Theorem, from {\it MathWorld--A Wolfram Web Resource}, \url{https://mathworld.wolfram.com/ButterflyTheorem.html}.

\bibitem{4b} Bogomolny, A.: The Butterfly Theorem, Interactive Mathematics Miscellany and Puzzles, \url{https://www.cut-the-knot.org/pythagoras/Butterfly.shtml}.

\bibitem{4c} Bogomolny, A.: A Better Butterfly Theorem, Interactive Mathematics Miscellany and Puzzles, \url{https://www.cut-the-knot.org/pythagoras/BetterButterfly.shtml}.

\bibitem{4d} Tienhoven, C. V.: Encyclopedia of Quadri-Figures, \url{https://chrisvantienhoven.nl/mathematics/encyclopedia}.

\bibitem{4e} Quadri Figures Group, \url{https://groups.io/g/Quadri-Figures-Group}.

\bibitem{4f} Euclidean Geometry Group, \url{https://groups.io/g/euclid}.

\bibitem{5} Tran, Q. H. and Gonz\'alez, L.: Problem 12147, {\it Amer. Math. Monthly}, {\bf 126:10}(2019) p.~946.

\bibitem{5a} Tran, Q. H.: Another synthetic proof of the butterfly theorem using the midline in triangle, {\it Forum. Geom.}, {\bf 16}(2016), p.~345--346.

\bibitem{6} Klamkin, M. S.: An extension of the butterfly theorem, {\it Math. Mag.}, {\bf 38}(1965), p.~206--208.

\bibitem{7} Sledge, J.: A generalization of the butterfly theorem, {\it J. of Undergraduate Math.}, {\bf 5}(1973), p.~3--4.

\bibitem{8} Volenec, V.: A generalization of the butterfly theorem, {\it Math. Commun.}, {\bf 5}(2000), p.~157--160.

\bibitem{9} Sliepcevi\'c, A.: A New Generalization of the
Butterfly Theorem, {\it J. Geom. Graph.}, {\bf 6}(2002), p.~61–68.

\bibitem{10a} Geogebra Geometry, \url{https://www.geogebra.org/geometry}.

\bibitem{10} Sage Notebook v6.10, \url{http://sagemath.org}.
\end{thebibliography}
\end{document}